\documentclass[12pt]{amsart}

\usepackage{amssymb}
\usepackage{mathrsfs}

\newtheorem{thm}{Theorem}[section]

\newtheorem{lem}[thm]{Lemma}

\newtheorem{prop}[thm]{Proposition}

\theoremstyle{remark}
\newtheorem{remark}[thm]{Remark}

\newcommand\sig{\operatorname{sig}}
\newcommand\mub{\mathfrak{M}}
\newcommand\Sfrak{\mathfrak{S}}
\newcommand\Fcal{\mathcal{F}}
\newcommand\Rbb{\mathbb{R}}
\newcommand\Nbb{\mathbb{N}}
\newcommand\Acal{\mathcal{A}}
\newcommand\Bcal{\mathcal{B}}
\newcommand\Gcal{\mathcal{G}}
\newcommand\sbf{\mathbf{s}}

\newcommand\swapvector{\delta}

\title
{Maximal unbalanced families}



\thanks{
The research presented in this paper was partially supported by an REU supplement to
NSF grant DMS--0757507.
The fourth author received support though the Tanner Dean Scholars program at Cornell University.
Any opinions, findings, and conclusions or recommendations expressed in this article
are those of the authors
and do not necessarily reflect the views of the National Science Foundation. 
}

\author[Billera et al.]{L.J.\ Billera}

\author[]{J.\ Tatch Moore}

\author[]{C.\ Dufort Moraites}

\author[]{Y.\ Wang}

\author[]{K.\ Williams}

\address{Louis Billera, Justin Moore \\
Department of Mathematics \\
Cornell University \\
Ithaca, NY 14853-4201 }

\begin{document}

\begin{abstract}
A family of subsets of the set $\{1,2,\dots,n\}$ is said to be {\em unbalanced} if the convex hull of its characteristic vectors misses the diagonal in the $n$-cube.
The purpose of this article is to develop the combinatorics of
maximal unbalanced families.
Specifically, we will prove lower and upper bounds on the number
of maximal unbalanced families of subsets of an $n$-element set -- both
bounds are of the form $2^{C n^2}$ for some $C > 0$. 
These families correspond to the chambers of a hyperplane arrangement,
the restricted all-subset arrangement, that has
arisen in various forms in physics, economics and psychometrics.
In particular, our bounds answer a question posed in thermal field theory concerning the order of the number of chambers of this arrangement. 
\end{abstract}

\maketitle

\section{Introduction}

Let $n$ be a natural number.
A family $\Fcal$ of subsets of $[n] = \{1,2,\dots,n\}$ is \emph{balanced} if there is a
convex combination of their characteristic functions which is constant;
a family is \emph{unbalanced} otherwise.
There is also an equivalent characterization of being unbalanced, provided by the Hahn-Banach Separation Theorem:
$\Fcal$ is unbalanced exactly when there is a $v$ in $\Rbb^n$ such that
\[
\sum_{i =1}^n v_i = 0
\]
and for all $F$ in $\Fcal$,
\[
\sum_{i \in F} v_i > 0.
\]

While minimal balanced families play an important role in game theory and
have been well studied in the literature (see, e.g., \cite{Billera,Peleg,Shapley}),
maximal unbalanced families have only been considered much more recently.
The second author's interest in them arose from \cite{tiling_curv_amen,amen_ramsey}, but they have
equivalent manifestations in psychometrics and economics \cite{unfold_co-dim1,Terao2}, as well as thermal field theory in physics
 \cite{what_calc_TFT}.
In all cases this can be seen by recognizing that maximal unbalanced families correspond
to the chambers of the hyperplane arrangement defined by the linear forms
\begin{equation}\label{rest_all_subsets}
\sum_{i \in F} v_i \qquad (F \subseteq [n], F \ne \emptyset, [n])
\end{equation}
in the vector space $V = \{v \in \Rbb^n : \sum_{i \in [n]} v_i = 0\}$.  In \cite{unfold_co-dim1,Terao2}, this is referred to as the {\em restricted all-subset arrangement}.

T. S. Evans \cite{what_calc_TFT} appears to be the first to have considered (an equivalent form of) this hyperplane arrangement -- see \eqref{all subsets} below -- and calculated the number of chambers for small values of $n$ (see also \cite{unfold_co-dim1,vanEj}).
We will use $E_n$ to denote the number of maximal unbalanced families of subsets of an $n$ element set.
\begin{prop} \cite{what_calc_TFT, unfold_co-dim1, vanEj}\footnote{
The value of $E_9$ does not appear in \cite{what_calc_TFT} but was rather taken from the On-Line Encyclopedia of Integer Sequences.
Our definition of $E_n$ follows the convention in \cite{what_calc_TFT};
the indices in the entries in OEIS are shifted by 1 at the time of this writing.
}
The first values of $E_n$ are $E_1=0$, $E_2 = 2$, $E_3 = 6$, $E_4 = 32$, $E_5 = 370$,
$E_6 =  11,292$, $E_7 = 1,066,044$, $E_8 = 347,326,352$, $E_9 = 419,172,756,930$.
\end{prop}

The main result of this paper are the following bounds on $E_n$, answering a question asked by Evans \cite{what_calc_TFT}.

\begin{thm}\label{mainthm}
For every $n \in \Nbb$,
\[
2^{\frac{(n-1)(n-2)}{2}} < E_n < 2^{(n-1)^2}.
\]
In particular, ${\displaystyle \lim_{n \to \infty} \frac{E_n}{n!} = \infty}$. 
\end{thm}

While the lower bound is obtained by applying known results from matroid theory, the upper bound involved developing
an understanding of the collection $\mub_n$ of maximal unbalanced families on an $n$-element set, which may be of independent
interest.


We will adopt the following notational conventions in this note.
Throughout this note, $i$, $j$, $k$, $l$, $m$, and $n$ will always be assumed to represent natural numbers (excluding $0$).
All counting will start at 1.
If $n$ is a natural number, then $2^{[n]}$ will be identified with the collection of all binary strings
of length $n$ as well as the collection
of all subsets of $[n]$ by associating a set $A \subseteq [n]$ with
its characteristic function.
We will use $\mub_n$ to denote the collection of all maximal unbalanced families of subsets of $[n]$.

Section \ref{lowerbound} is devoted to the derivation of the lower bound, while \S \ref{upperbound} will treat the upper bound.  We note that the results in \S \ref{lowerbound} benefitted from discussions with E.\ Swartz.

\section{The lower bound\label{lowerbound}}
To derive a lower bound on the number of maximal unbalanced families, and thus on the number of chambers of the arrangement in \eqref{rest_all_subsets}, we will need a combinatorially equivalent form of \eqref{rest_all_subsets} given by the linear forms
\begin{equation}\label{all subsets}
\sum_{i \in F} v_i \qquad (F \subseteq [n-1], F \ne \emptyset)
\end{equation}
in $\Rbb^{n-1}$.  In other words, we consider the arrangement $\mathcal{A}_{n}$ in $\Rbb^{n-1}$ consisting of the $2^{n-1}-1$ hyperplanes having as normals all nonzero 0-1 vectors in $\Rbb^{n-1}$.
In order to count the number of components of $\Rbb^{n-1}\setminus\mathcal{A}_{n}$,
we apply a result of Zaslavsky \cite{Zas}.

To do this, we define the {\em lattice of flats} ${L}_{n}$ of the arrangement $\mathcal{A}_{n}$ to be the family of all subspaces spanned over $\mathbb{Q}$ by subsets of the set of nonzero 0-1 vectors in $\Rbb^{n-1}$, ordered by inclusion.  The rank of $L_{n}$ is $n-1$.  The {\em characteristic polynomial} of $\mathcal{A}_{n}$ is then
\begin{equation}\label{charpoly}
\chi(\mathcal{A}_{n},t)= \sum_{x\in L_{n}} \mu(0, x)~ t^{\text{rank}(L_{n})-\text{rank}(x)}=\sum_{k=0}^{n-1}  w_{k}(L_{n}) ~t^{n-1-k},
\end{equation}
where $\mu$ is the M\"obius function of $L_{n}$.  The quantities $w_{k}(L_{n})$ in \eqref{charpoly} are called the {\em Whitney numbers of the first kind}.

The result of Zaslavsky \cite[Theorem A]{Zas} is that the number of chambers of $\mathcal{A}_{n}$ is 
\begin{equation}\label{numberofchambers}
(-1)^{n-1}\chi(\mathcal{A}_{n}, -1)=\sum_{x\in L_{n}} |\mu(0,x)|= \sum_{k=0}^{n-1} |w_{k}(L_{n})|.
\end{equation}
Unfortunately, we do not have an explicit formula for the polynomial $\chi(\mathcal{A}_{n},t)$.

To give a lower bound for the number of chambers of $\mathcal{A}_{n}$, we consider the linear matroid of all subspaces spanned over the 2-element field $\mathbb{F}_{2}$ by these same 0-1 vectors, now considered to be the set $\mathbb{F}_{2}^{n-1}\setminus \{(0,0,\dots,0)\}$.  By abuse of notation, we will denote this matroid by $\mathcal{A}_{n}^{(2)}$ and its lattice of flats by $L_{n}^{(2)}$.  The rank of $L_{n}^{(2)}$ is again $n-1$.

Since independence over $\mathbb{F}_{2}$ implies independence over $\mathbb{Q}$, we have that the map $\mathcal{A}_{n} \rightarrow \mathcal{A}_{n}^{(2)}$ is a rank-preserving weak map, and so, by a theorem of Lucas \cite[Proposition 7.4]{Lucas} (see also \cite[Corollary 9.3.7]{weak_maps}), we obtain
$$|w_{k}(\mathcal{A}_{n})| \ge |w_{k}(\mathcal{A}_{n}^{(2)})|$$ for each $k$, and so we  conclude
\begin{equation}\label{compareZ2}
(-1)^{n-1}\chi(\mathcal{A}_{n}, -1) \ge (-1)^{n-1}\chi(\mathcal{A}_{n}^{(2)}, -1).
\end{equation}

To complete the bound, we observe that $\mathcal{A}_{n}^{(2)}$ is the $(n-1)$-dimensional projective geometry over $\mathbb{F}_{2}$, and so its characteristic polynomial (see, for example, \cite[Example 3.6(3)]{Brylawski}) is
\begin{equation}\label{projcharpoly}
\chi(\mathcal{A}_{n}^{(2)},t) = \prod_{i=0}^{n-2} (t-2^{i}).
\end{equation}
Together, \eqref{numberofchambers}, \eqref{compareZ2} and \eqref{projcharpoly} give us the lower bound in Theorem \ref{mainthm}.

\begin{thm}
The number of maximal unbalanced families in $[n]$, equivalently, the number of chambers of the arrangement $\mathcal{A}_{n}$, is at least $\prod_{i=0}^{n-2} (2^{i}+1)$.
Thus 
$$E_{n} > \prod_{i=0}^{n-2} 2^{i} = 2^{\frac{(n-1)(n-2)}{2}}.$$
\end{thm}

\section{The signature of a maximal unbalanced family\label{upperbound}}

Let us begin by making the following easy observations.
If $\Fcal$ is in $\mub_n$, then $\Fcal$ does not contain either $\emptyset$ or $[n]$.
On the other hand, if $F$ is any other subset of $[n]$, then exactly one of $F$ and $[n] \setminus F$ are in $\Fcal$.
This follows from the fact that if $\Fcal$ is in $\mub_n$, then there is a $v \in \Rbb^n$
such that if $F$ is in $\Fcal$, then $\sum_{i \in F} v_i > 0$.
If we choose such a $v$ to be in generic position, then $\sum_{i \in X} v_i \ne 0$ unless $X$ is
$\emptyset$ or $[n]$.
Thus if $\Fcal$ is maximal, it selects between every nontrivial subset of $[n]$ and its complement, so $\left\vert \Fcal \right\vert = 2^{n-1} - 1$ for every $\Fcal \in \mub_n$.

It will be useful to let $\Sfrak_n$ denote the collection of all families $\Acal$ of subsets of $[n]$ such that:
\begin{itemize}

\item neither $\emptyset$ nor $[n]$ are in $\Acal$;

\item if $A$ is a proper nonempty subset of $[n]$,
then exactly one of $A$ and $[n] \setminus A$ are in $\Acal$.

\end{itemize}
Note that $\mub_n \subset \Sfrak_n$, and the inclusion is strict for $n \ge 3$.

The collection $\mub_n$ is equipped with a natural notion of adjacency:
$\Fcal$ and $\Gcal$ are adjacent if $\Fcal \setminus \Gcal$ has one element.
If we view elements of $\mub_n$ as chambers in the hyperplane arrangement $\Acal_n$, then
this notion of adjacency coincides with the adjacency of chambers.
Using the perspective provided by the hyperplane arrangement, it should be clear that
the adjacency graph on $\mub_n$ is connected.

If $\Fcal$ is in $\Sfrak_n$ for some $n$, define the \emph{signature} of $\Fcal$ -- denoted
$\sig(\Fcal)$ -- to be the sequence $\sbf$ of length $n$ defined by
\[
s_i = \left\vert\{F \in \Fcal : i \in F\}\right\vert.
\]
The goal of this section is to prove that the signature map is injective on $\mub_n$ for each $n$.
Moreover we will show that the signature of an element of $\mub_n$ can never coincide with the signature
of an element of $\Sfrak_n \setminus \mub_n$.
We will also show that the parity of the entries of $\sig(\Fcal)$ are always the same.

These observations are already enough to yield the upper bound  in Theorem \ref{mainthm}.  To see this, observe that
there are fewer than $(2^{n-1})^n$ possible signatures of a family with $2^{n-1}-1$ elements (for example, $(0,0,\dots,0)$ cannot be a signature).
If we require that all entries are even or all are odd, there are fewer than
$(2^{n-1})^n/2^{n-1} = 2^{(n-1)^2}$ such signatures, proving the upper bound.

To verify the above claims, define $\swapvector_{F} \in \{-1,1\}^{n}$, for $F$ a nonempty proper subset of $[n]$ and $i < n$, by
\[
\swapvector_F(i) = 
\begin{cases}
1 & \text{ if } i \in F \\
-1 & \text{ if } i \not \in F.\\
\end{cases}
\]
If $\Fcal$ is a family of nonempty proper subsets of $n$, define 
\[
\swapvector_\Fcal = \sum_{F \in \Fcal} \swapvector_F.
\]

\begin{lem} \label{unbal-lemma}
If $\Fcal$ is a nonempty unbalanced family of subsets of $[n]$, then $\swapvector_\Fcal$ is not constant.
\end{lem}

\begin{proof}
Suppose that this is not the case and 
notice that the cardinalities of
\[
\{F \in \Fcal : i \in F\}
\]
\[
\{F \in \Fcal : i \not \in F\}
\]
do not depend on $i$;
their difference is the constant entry of $\swapvector_\Fcal$ and
their sum is the cardinality of $\Fcal$.
This means, however, that the signature of $\Fcal$ is constant and,
in particular, that the uniform probability measure on $\Fcal$ witnesses that $\Fcal$ is balanced, a contradiction.
\end{proof}

\begin{thm} \label{signatures-unique}
The function which takes an element of $\Sfrak_n$ to its signature is one-to-one on $\mub_{n}$.  Furthermore, the signature of an element of $\mub_n$ can never coincide with the signature of balanced family in $\Sfrak_n$.
\end{thm}

\begin{proof}
Suppose that $\Acal$ and $\Bcal$ are distinct elements of $\Sfrak_n$ and that $\Bcal$ is unbalanced.
Observe that $\Acal$ and $\Bcal$ have the same signature if and only if
$\swapvector_{\Bcal \setminus \Acal} = ( 0, \ldots, 0 )$:
as $\Acal$ and $\Bcal$ differ by a series of swaps, $\Bcal \setminus \Acal$ is the family of swapped sets.
Hence, $\sig(\Bcal) = \sig(\Acal) + \swapvector_{\Bcal \setminus \Acal}$.
Since $\Bcal$ is unbalanced, so is $\Bcal \setminus \Acal$, so Lemma \ref{unbal-lemma} implies that $\swapvector_{\Bcal \setminus \Acal}$ is not constant and, in particular, is not identically $0$.
Consequently, $\Acal$ and $\Bcal$ have distinct signatures.
\end{proof}

\begin{prop}
If $\Fcal$ is an element of $\mub_n$, then either all entries of $\sig(\Fcal)$ are even or all entries
are odd.
\end{prop}

\begin{proof}
If $\Fcal$ is the family of all nonempty subsets of $[n]$ that do not contain $1$,
then $\sig(\Fcal)$ is the sequence $(0,2^{n-2},\ldots,2^{n-2})$.
In particular, the conclusion of the proposition holds for $\Fcal$.
Next observe that if $\Gcal_0$ and $\Gcal_1$ are adjacent elements of $\mub_n$,
then every coordinate of $\sig(\Gcal_0)$ differs by $\pm 1$ from the corresponding coordinate of
$\sig(\Gcal_1)$.
The proposition now follows from the connectedness of the adjacency graph on $\mub_n$.
\end{proof}

\begin{remark}
The above proof actually shows that the adjacency graph on $\mub_n$ is bipartite.
\end{remark}

\begin{remark}
Maximal unbalanced familes might be viewed in the context of {\em threshold families} whose defining weights sum to zero.  See for example \cite{KlivRein}, where attention is restricted to uniform families of subsets ({\em i.e.}, all subsets having the same cardinality), and the signature of a family is called its {\em degree sequence}.  Our Theorem \ref{signatures-unique} should be compared to the unique realizability conclusion of \cite[Theorem 3.1]{KlivRein}.  For our particular case, the conclusion of Theorem \ref{signatures-unique} is not directly comparable to the latter, since we have not restricted to uniform families, and the proof here is more elementary.
\end{remark}


\begin{thebibliography}{1}

\bibitem{Billera}
Louis J. Billera.
\newblock On games without side payments arising from a general class of markets.
\newblock {\em J. Math. Econom.} {\bf 1} (1974), no. 2, 129--139. 

\bibitem{tiling_curv_amen}
Jonathan Block and Shmuel Weinberger.
\newblock Aperiodic tilings, positive scalar curvature and amenability of
  spaces.
\newblock {\em J. Amer. Math. Soc.}, 5(4):907--918, 1992.

\bibitem{Brylawski}
Tom Brylawski.
\newblock Intersection theory for embeddings of matroids into uniform geometries. 
\newblock {\em Stud.\ Appl.\ Math.\ } {\bf 61} (1979), no. 3, 211--244. 

\bibitem{what_calc_TFT}
T.~S. Evans.
\newblock What is being calculated with {T}hermal {F}ield {T}heory?
\newblock In A.~Astbury, B.~A. Campbell, W.~Israel, F.~C. Khanna, D.~Page, and
  J.~L. Pinfold, editors, {\em Particle Physics and Cosmology - Proceedings of
  the Ninth Lake Louise Winter Institute}, pages 343--352. World Scientific,
  1995.

\bibitem{unfold_co-dim1}
Hidehiko Kamiya, Akimichi Takemura, and Hiroaki Terao.
\newblock Ranking patterns of unfolding models of codimension one.
\newblock {\em Adv. in Appl. Math.}, {\bf 47(2)}, 379--400, 2011.

\bibitem{Terao2}
Hidehiko Kamiya, Akimichi Takemura, and Hiroaki Terao.
\newblock Arrangements stable under Coxeter groups.
\newblock ArXiv preprint 1103.5179v2, October 2011.

\bibitem{KlivRein}
C.J.\ Klivans and V. Reiner,
\newblock Shifted set families, degree sequences, and plethysm. 
\newblock {\em Electron.\ J.\ Combin.}\ {\bf 15} (2008), no. 1, Research Paper 14, 35 pp. 

\bibitem{weak_maps}
Joseph P.S.\ Kung and Hien Q.\ Nguyen.
\newblock Weak maps.
\newblock Chapter 9 in Neil White, editor, {\em Theory of Matroids}, pages 254--271.  Encyclopedia of Mathematics and it Applications, Vol. 26, Cambridge University Press, Cambridge, 1986.

\bibitem{Lucas}
Dean Lucas.
\newblock Weak maps of combinatorial geometries.
\newblock {\em Trans.\ Amer.\ Math.\ Soc.\ } {\bf 206} (1975), 247--279.

\bibitem{amen_ramsey}
Justin~Tatch Moore.
\newblock Amenability and {R}amsey theory.
\newblock ArXiv preprint 1106.3127, June 2011.

\bibitem{Peleg}
Bezalel Peleg.
\newblock An inductive method for constructing minimal balanced collections of finite sets.
\newblock {\em Naval Res.\ Logist.\ Quart.\ } {\bf 12} (1965), 155--162.

\bibitem{Shapley}
Lloyd Shapley.
\newblock  On balanced sets and cores.
\newblock {\em Naval Res.\ Logist.\ Quart.\ } {\bf 14} (1967), 453--460.  

\bibitem{vanEj}
M.\ van Eijck.
\newblock {\em Thermal Field Theory and Finite-Temperature Renormalisation Group.}
\newblock PhD thesis, Univ. Amsterdam, 4th Dec. 1995.

\bibitem{Zas}
Thomas Zaslavsky.
\newblock Facing up to arrangements: face-count formulas for partitions of space by hyperplanes.
\newblock {\em Mem. Amer. Math. Soc.} {\bf 1} (1975), issue 1, number 154, vii+102 pp.

\end{thebibliography}


\def\cprime{$'$} \def\Dbar{\leavevmode\lower.6ex\hbox to 0pt{\hskip-.23ex
  \accent"16\hss}D}

\end{document}